\documentclass[letterpaper, 10 pt, conference]{ieeeconf}

\IEEEoverridecommandlockouts
\renewcommand{\baselinestretch}{0.982}

\usepackage{float}
\usepackage{paralist}
\usepackage{amsfonts}
\usepackage{graphics}
\usepackage{epsfig}
\usepackage{mathptmx}
\usepackage{times}
\usepackage{amsmath}
\usepackage{amssymb}
\usepackage{physics}
\usepackage{mdwmath}
\usepackage{mdwtab}
\usepackage{color}
\usepackage[hidelinks]{hyperref}
\usepackage{algpseudocode}
\usepackage{algorithm}
\usepackage{caption}
\usepackage{subcaption}
\usepackage{bm}
\usepackage{tikz}
\usetikzlibrary{shapes, arrows.meta, positioning}
\usepackage{booktabs}
\usepackage{tabularx}
\usepackage{multirow}
\usepackage{subcaption}

\newtheorem{theorem}{Theorem}

\newtheorem{proposition}{Proposition}

\newtheorem{remark}{Remark}
\newtheorem{problem}{Problem}
\newtheorem{corollary}{Corollary}
\newtheorem{assumption}{Assumption}

\begin{document}

\title{\LARGE \bf
Rank-Dependent Error Bounds and Near-Optimality in Quantum Control via Hierarchical Tucker Surrogates
}

\author{Nahid Binandeh Dehaghani, Rafal Wisniewski, Julian Berberich, A. Pedro Aguiar
 \thanks{N. Dehaghani and R. Wisniewski are with the Department of Electronic Systems, Aalborg University, Fredrik Bajers vej 7c, DK-9220 Aalborg, Denmark
        {\tt\small \{nahidbd,raf\}@es.aau.dk}}%
\thanks{J. Berberich is 
with the Institute for Systems Theory and Automatic Control and the Center for Integrated Quantum Science and Technology (IQST), University of Stuttgart, 70569 Stuttgart, Germany
 {\tt\small julian.berberich@ist.uni-stuttgart.de}
}        
 \thanks{A. Pedro Aguiar is with the Research Center for Systems and Technologies (SYSTEC), Electrical and Computer Engineering Department, FEUP - Faculty of Engineering, University of Porto, Rua Dr. Roberto Frias sn, i219, 4200-465 Porto, Portugal
         {\tt\small pedro.aguiar@fe.up.pt}}
\thanks{The authors acknowledge the support of the Danish e-Infrastructure Consortium (DeiC) and the National Quantum Algorithm Academy (NQAA) through the Postdoctoral Scholarship under the project ``Quantum-Driven Solutions for Multi-Agent Systems and Advanced Computation''. This work was also funded by QuantERA FeedbacQ, by
the Deutsche Forschungsgemeinschaft (DFG, German Research Foundation) – 579821331 and 583407438, and cofunded by the European Commission.
Additional support was provided by the Research Center for Systems and Technologies (SYSTEC, 10.54499/UID/00147/2025) and the Associate Laboratory Advanced Production and Intelligent Systems (ARISE, 10.54499/LA/P/0112/2020) funded by Fundação para a Ciência e a Tecnologia, I.P./ MCTES through national funds.}
}
\maketitle

\begin{abstract}
We develop a certified finite-horizon quantum optimal-control framework
based on fixed-rank Hierarchical Tucker (HT) surrogates. Under a uniform
rank-dependent HT truncation-accuracy condition, we establish exponentially
decaying trajectory, cost, and value-function errors and near-optimality of
surrogate-generated controls. We further derive a logarithmic
rank--performance relation, an a posteriori cost certificate from realized
truncation perturbations, and guarantees for inexact surrogate optimization
and control convergence under additional regularity conditions. Numerical
experiments on controlled XXZ spin chains provide finite-sample evidence
consistent with the truncation condition and illustrate the trade-off
between HT rank, representation size, surrogate accuracy, and control
performance.
\end{abstract}

\section{Introduction}

Quantum optimal control provides a 
systematic 
framework for steering
quantum systems toward desired states or operations under physical
constraints, with applications in quantum information processing,
quantum simulation, spectroscopy, and atomic and molecular control
\cite{glaser2015training,dalessandro2021introduction}. For interacting
many-body systems, however, the Hilbert-space dimension grows exponentially
with system size, making repeated state propagation and optimization
computationally demanding.
A variety of quantum-control methods have been developed to construct
high-fidelity protocols \cite{glaser2015training,rabitz2008quantum}.
Their application to large many-body systems remains limited by the cost of
representing and propagating the quantum state. This has motivated low-rank
tensor-network representations, which provide compact descriptions when
the corresponding ranks remain moderate. In particular, matrix-product-state
methods have enabled control studies beyond direct full-state simulation
\cite{jensen2021achieving,metz2023selfcorrecting}.

The Hierarchical Tucker (HT) format provides a flexible tree-structured
decomposition of high-dimensional tensors. Hierarchical singular value
decomposition (HSVD) enables systematic rank truncation with approximation
estimates based on discarded hierarchical singular values
\cite{grasedyck2010hierarchical,hackbusch2012tensor}. Rapidly decaying
hierarchical spectra can therefore permit accurate fixed-rank
approximations at moderate HT ranks.
In our previous work \cite{dehaghani2026certified}, fixed-rank HT
truncation was studied from a closed-loop robustness perspective, yielding
practical-stability and rank-dependent tracking guarantees. That analysis
did not address finite-horizon optimal control or quantify the performance
loss when a control optimized on an HT surrogate is applied to the
full-order system.

This distinction is important because accurate state approximation alone
does not guarantee accurate control design: a reduced model may approximate
the dynamics well while its optimized control remains suboptimal on the
full-order system. Thus, quantitative guarantees linking tensor rank,
dynamical approximation, and optimal-control performance are needed.
This letter establishes such a connection for finite-horizon quantum
optimal control using fixed-rank HT surrogates. Under a uniform
rank-dependent HT truncation-accuracy condition, we relate HT rank to
trajectory, objective, and value-function errors and to the full-order
performance of surrogate-generated controls.

The main contributions are:
(i) explicit rank-dependent bounds on trajectory, cost-functional, and
value-function approximation errors;
(ii) a near-optimality guarantee for surrogate-generated controls, showing
exponential decay of the full-order optimality gap with HT rank under the
assumed exponential truncation-accuracy condition;
(iii) a logarithmic rank--performance relation providing a sufficient rank
for a prescribed full-order performance tolerance, together with an
a posteriori certificate based on realized truncation perturbations; and
(iv) extensions to inexact surrogate optimization and, under local
quadratic growth and uniqueness of the full-order minimizer, convergence
of surrogate-optimal controls toward the full-order optimizer.

\section{Optimal Control Problem}



We consider the semi-discrete controlled Schr\"odinger model introduced in
\cite{dehaghani2026certified}, namely
\begin{equation}
i\hbar \dot{\Psi}(t)
=
\left(H_0^N+u(t)H_1^N\right)\Psi(t),
\label{eq:schrodinger}
\end{equation}
where $\Psi(t)\in\mathbb C^N$ denotes the quantum state,
$u(t)\in\mathbb R$ is a scalar control input, and
$H_0^N,H_1^N\in\mathbb C^{N\times N}$ are finite-dimensional
discretizations of the drift and control Hamiltonians, respectively.
The superscript $N$ indicates dependence on the discretization level.
Following \cite{dehaghani2026certified}, we assume that $H_0^N$ and
$H_1^N$ are Hermitian discretizations obtained from standard symmetric
schemes with compatible boundary conditions. Consequently, the resulting
Schr\"odinger evolution is unitary and preserves the state norm.
Throughout, $\|\cdot\|$ denotes the Euclidean norm on $\mathbb C^N$.
For numerical optimization, the control is assumed piecewise constant over
sampling intervals of length $\Delta t>0$. Denoting by $u_k$ the control
applied over $[k\Delta t,(k+1)\Delta t)$, the resulting discrete-time
dynamics are
\begin{equation}
\Psi_{k+1}
=
F_{\Delta t}(\Psi_k,u_k),
\quad
\Psi_0=x_0,
\quad
k=0,\ldots,N_T-1,
\label{eq:discrete_flow}
\end{equation}
where
$
F_{\Delta t}(\Psi,u)
=
\exp\!\left(
-\frac{i\Delta t}{\hbar}
(H_0^N+uH_1^N)
\right)\Psi .
\label{eq:flow_map}
$
Since $H_0^N+uH_1^N$ is Hermitian for every admissible $u$, the flow map
is unitary and therefore non-expansive:
$
\|F_{\Delta t}(x,u)-F_{\Delta t}(y,u)\|
=
\|x-y\|.
\label{eq:nonexpansive}
$
Let
$\mathcal U=\{u\in\mathbb R:|u|\leq u_{\max}\}$
denote the admissible control set. A control sequence over a horizon of
$N_T$ time steps is written as
$\mathbf u=(u_0,\ldots,u_{N_T-1})\in\mathcal U^{N_T}$.
For $x_0\in\mathbb C^N$ and
$\mathbf u\in\mathcal U^{N_T}$, the corresponding state trajectory
generated by~\eqref{eq:discrete_flow} is denoted by
$\{\Psi_k(x_0,\mathbf u)\}_{k=0}^{N_T}$.


For the trajectory generated by~\eqref{eq:discrete_flow}, we consider
finite-horizon optimal-control problems with cost functional
\begin{equation}
J(x_0,\mathbf u)
=
\sum_{k=0}^{N_T-1}
L\!\left(\Psi_k(x_0,\mathbf u),u_k\right)
+
\Phi\!\left(\Psi_{N_T}(x_0,\mathbf u)\right),
\label{eq:cost_functional}
\end{equation}
where $L:\mathbb C^N\times\mathcal U\rightarrow\mathbb R$ is a running
cost and $\Phi:\mathbb C^N\rightarrow\mathbb R$ is a terminal cost.
The running cost may penalize state deviations and/or control effort,
while the terminal cost encodes the terminal control objective. The
formulation~\eqref{eq:cost_functional} encompasses a broad class of
quantum optimal-control problems \cite{dehaghani2023quantum}, including
state transfer, trajectory tracking, and minimum-energy control. A typical
state-transfer objective is
$\Phi(\Psi)=1-|\langle\Psi,\Psi_{\rm tar}\rangle|^2$, where
$\Psi_{\rm tar}\in\mathbb C^N$ is a normalized target state. Control effort
is often penalized through a quadratic running cost of the form
$L(\Psi,u)=\rho|u|^2$, with $\rho>0$. The subsequent analysis is not
restricted to these choices and applies to general cost functions
satisfying the following regularity condition.

\begin{assumption}[Lipschitz Cost Regularity] \! \! \! \! \!
There exist constants $L_\Phi>0$ and $L_L\geq0$ such that
$
|\Phi(x)-\Phi(y)|
\leq
L_\Phi\|x-y\|$, $|L(x,u)-L(y,u)|
\leq
L_L\|x-y\|,
$
for all admissible states $x,y$ and controls $u$. The running-cost
condition is imposed only with respect to the state variable and is
required to hold uniformly over admissible controls; no Lipschitz
regularity with respect to the control input is needed.
\label{ass:cost_regularity}
\end{assumption}

The fidelity-based terminal cost introduced above satisfies
Assumption~\ref{ass:cost_regularity} on any bounded state set. If $\|x\|,\|y\|\leq M$ and $\|\Psi_{\rm tar}\|=1$, then
$
|\Phi(x)-\Phi(y)|
\leq
2M\|x-y\|.
$
In particular, for normalized quantum states $(M=1)$, one may take
$L_\Phi=2$.





The full-order finite-horizon optimal control problem is
$
\inf_{\mathbf u\in\mathcal U^{N_T}}
J(x_0,\mathbf u)
$
subject to the discrete-time Schr\"odinger dynamics
\eqref{eq:discrete_flow}. The associated value function is defined by
$
V(x_0)
=
\inf_{\mathbf u\in\mathcal U^{N_T}}
J(x_0,\mathbf u).
$
Whenever a minimizer exists, we denote any such minimizer by
$
\mathbf u^\star
\in
\arg\min_{\mathbf u\in\mathcal U^{N_T}}
J(x_0,\mathbf u),
$
and denote the corresponding optimal trajectory by
$\Psi_k^\star=\Psi_k(x_0,\mathbf u^\star)$,
$k=0,\ldots,N_T$.

\begin{problem}
Construct a fixed-rank Hierarchical Tucker surrogate model and
establish explicit rank-dependent bounds on the value-function
error and on the loss of optimality incurred by controls optimized
on the HT surrogate when evaluated on the full-order dynamics.
\end{problem}

\section{Hierarchical Tucker Surrogate Optimal Control}

This section introduces the fixed-rank Hierarchical Tucker surrogate model
used to approximate the full-order Schr\"odinger dynamics and formulate the
corresponding surrogate optimal control problem.

\subsection{Hierarchical Tucker Representation and Fixed-Rank Truncation}
%
%
We identify the state $\Psi\in\mathbb C^N$ with an order-$n$
tensor through the tensor-product factorization
$\mathbb C^N\cong\bigotimes_{i=1}^n\mathbb C^{d_i}$,
$\prod_{i=1}^n d_i=N$.
A dimension tree $\mathcal T$ hierarchically partitions the mode set
$\{1,\ldots,n\}$ and thereby specifies the bipartitions used in the
HT representation; see \cite[Def.~3.1]{grasedyck2010hierarchical}.
We fix $\mathcal T$ throughout the analysis and denote by
$\mathcal T^\circ$ its set of non-root nodes.
For each $t\in\mathcal{T}^{\circ}$, let
$\bar t=\{1,\ldots,n\}\setminus t$ and denote by
$X^{(t)}(\Psi)$ the $t$-matricization of $\Psi$, whose row indices
correspond to the modes in $t$ and whose column indices correspond to
the complementary modes in $\bar t$; see
\cite[Def.~3.3]{grasedyck2010hierarchical}.
Let
$\{\sigma_\alpha^{(t)}(\Psi)\}_{\alpha\geq1}$ denote the singular
values of $X^{(t)}(\Psi)$. The hierarchical Tucker rank at node $t$ is
$
r_t(\Psi)=\operatorname{rank}\!\left(X^{(t)}(\Psi)\right).
$
Given a uniform hierarchical rank budget $r\in\mathbb N$, we restrict
the hierarchical ranks according to
$
r_t\leq r
$
for all $t\in\mathcal T^\circ$.
The corresponding fixed-rank HT class is
$
\mathcal H_r
=
\left\{
\Psi:
\operatorname{rank}\!\left(X^{(t)}(\Psi)\right)\leq r,
\quad
\forall\,t\in\mathcal T^\circ
\right\}.
$
Let $\Pi_r:\mathbb C^N\rightarrow\mathcal H_r$ denote the HSVD
truncation map, obtained by truncating the hierarchical singular-value
decompositions associated with the prescribed dimension tree to rank
at most $r$ at each node $t\in\mathcal T^\circ$.

\begin{assumption}[Exponential Hierarchical Spectral Decay]
\label{ass:decay}
Let $\Psi\in\mathbb{C}^N$ be a state tensor. We say that $\Psi$
satisfies exponential hierarchical spectral decay if there exist
constants $C_1,c>0$ such that
$
\sigma_\alpha^{(t)}(\Psi)\leq C_1e^{-c\alpha},
$
for every $t\in\mathcal T^\circ$ and all $\alpha\geq1$.
\end{assumption}


Assumption~\ref{ass:decay} admits a natural interpretation in terms of
bipartite entanglement. For a normalized pure state, the hierarchical
singular values coincide with the Schmidt coefficients across the
corresponding tree-induced bipartitions
\cite{schollwock2011density,eisert2010colloquium}. Thus,
Assumption~\ref{ass:decay} requires exponentially decaying Schmidt spectra,
with uniformly bounded Schmidt rank as a sufficient special case. Bounded
entanglement entropy alone, however, does not generally imply exponential
Schmidt-spectrum decay.

The following approximation result is recalled from
\cite{dehaghani2026certified}.

\begin{proposition}\label{prop:approx}
Let $\Psi$ satisfy Assumption~\ref{ass:decay}. Then there exist constants
$C_2,c'>0$, independent of the prescribed rank $r$ and determined by
the spectral-decay constants and the fixed dimension tree, such that
$
\left\|\Psi-\Pi_r(\Psi)\right\|
\leq C_2 e^{-c'r}.
$
Here, $C_2$ is a rank-independent approximation constant and $c'$
denotes the effective exponential decay rate inherited from the
hierarchical spectral decay in Assumption~\ref{ass:decay}.
Thus, under Assumption~\ref{ass:decay}, increasing the prescribed HT
rank $r$ yields exponentially improving approximation accuracy, at the
expense of increasing the number of degrees of freedom in the HT
representation.
\end{proposition}


In particular, if $x_0$ satisfies Assumption~\ref{ass:decay},
Proposition~\ref{prop:approx} gives
$
\|x_0-\Pi_r(x_0)\|\leq C_2e^{-c'r},
$
which bounds the initial trajectory approximation error. Combined with
the uniform truncation-accuracy condition introduced below, this yields
rank-dependent trajectory and optimal-control performance bounds.

\subsection{HT-Surrogate Optimal Control Problem}

We construct a fixed-rank HT surrogate of the full-order
Schr\"odinger dynamics by applying HSVD truncation after each
time step. Given a control sequence
$\mathbf u=(u_0,\ldots,u_{N_T-1})\in\mathcal{U}^{N_T}$, the
HT-surrogate dynamics for $k=0,\ldots,N_T-1$ are defined by
\begin{equation}
\Psi_{k+1}^{(r)}
=
\Pi_r
F_{\Delta t}\!\left(\Psi_k^{(r)},u_k\right),
\qquad
\Psi_0^{(r)}=\Pi_r(x_0).
\label{eq:surrogate_dynamics}
\end{equation}
The corresponding surrogate trajectory is denoted by
$\{\Psi_k^{(r)}(x_0,\mathbf u)\}_{k=0}^{N_T}$.
Analogously to~\eqref{eq:cost_functional}, we define the HT-surrogate
cost functional by
$
J_r(x_0,\mathbf u)
\!
=
\!
\sum_{k=0}^{N_T-1}
L\!\left(\Psi_k^{(r)}(x_0,\mathbf u),u_k\right)
\!\!\!\!+\!\!\!\!
\Phi\!\left(\Psi_{N_T}^{(r)}(x_0,\mathbf u)\right).
$
The corresponding surrogate value function is
$
V_r(x_0)
=
\inf_{\mathbf u\in\mathcal{U}^{N_T}}
J_r(x_0,\mathbf u).
$
Whenever a minimizer exists, let $\mathbf u_r^\star$ denote any
surrogate minimizer, i.e.,
$
\mathbf u_r^\star
\in
\arg\min_{\mathbf u\in\mathcal{U}^{N_T}}
J_r(x_0,\mathbf u),
$
with associated surrogate optimal trajectory
$
\Psi_k^{(r)\star}
=
\Psi_k^{(r)}(x_0,\mathbf u_r^\star)
$,
$k=0,\ldots,N_T$.
The objective of the subsequent analysis is to establish rank-dependent
bounds on the value-function error
$
|V_r(x_0)-V(x_0)|
$
and on the full-order optimality gap of the surrogate-optimal control,
$
J(x_0,\mathbf u_r^\star)-V(x_0).
$

\subsection{Trajectory Approximation Under Fixed-Rank Truncation}

We compare the full-order and HT-surrogate trajectories under the same
admissible control sequence. The analysis requires uniform rank-dependent
truncation accuracy along the surrogate evolution.

\begin{assumption}[Uniform HT Truncation Accuracy]
\label{ass:compressibility}
There exist constants $C_{\rm tr},c_{\rm tr}>0$, independent of the
admissible control sequence $\mathbf u$, the time index $k$, and the
rank $r$, such that the pre-truncation surrogate states
$
Z_k(\mathbf u)
=
F_{\Delta t}\!\left(\Psi_k^{(r)}(\mathbf u),u_k\right)
$
satisfy
$
\|Z_k(\mathbf u)-\Pi_r(Z_k(\mathbf u))\|
\leq
C_{\rm tr}e^{-c_{\rm tr}r}
$
for every $\mathbf u\in\mathcal U^{N_T}$,
$k=0,\ldots,N_T-1$, and $r\geq1$.
\end{assumption}

\begin{proposition}[Tail-decay condition implying
Assumption~\ref{ass:compressibility}]
\label{prop:tail_decay}
Suppose there exist constants $C_{\rm tail},c_{\rm tail}>0$,
independent of $\mathbf u$, $k$, and $r$, such that
$\sum_{\alpha>r}(\sigma_\alpha^{(t)}(Z_k(\mathbf u)))^2
\leq C_{\rm tail}e^{-2c_{\rm tail}r}$
for every $\mathbf u\in\mathcal U^{N_T}$,
$k=0,\ldots,N_T-1$, $t\in\mathcal T^\circ$, and $r\geq1$.
Then Assumption~\ref{ass:compressibility} holds. In particular,
$\|Z_k(\mathbf u)-\Pi_rZ_k(\mathbf u)\|
\leq C_{\rm tr}e^{-c_{\rm tr}r}$,
where one may take $c_{\rm tr}=c_{\rm tail}$ and
$C_{\rm tr}=\sqrt{|\mathcal T^\circ|C_{\rm tail}}$.
\end{proposition}

\begin{proof}
By the HSVD truncation estimate~\cite{grasedyck2010hierarchical},
$\|Z_k-\Pi_rZ_k\|^2
\leq\sum_{t\in\mathcal T^\circ}\sum_{\alpha>r}
(\sigma_\alpha^{(t)}(Z_k))^2$.
Using the assumed uniform tail bound gives
$\|Z_k-\Pi_rZ_k\|^2
\leq|\mathcal T^\circ|C_{\rm tail}e^{-2c_{\rm tail}r}$.
Taking square roots proves the claim.
\end{proof}

\begin{remark}
A sufficient condition for the tail-decay hypothesis of
Proposition~\ref{prop:tail_decay} is uniform exponential
hierarchical spectral decay,
$\sigma_\alpha^{(t)}(Z_k(\mathbf u))
\leq C_{\rm sp}e^{-c_{\rm sp}\alpha}$.
Indeed,
$\sum_{\alpha>r}(\sigma_\alpha^{(t)}(Z_k(\mathbf u)))^2
\leq
C_{\rm sp}^2e^{-2c_{\rm sp}(r+1)}/(1-e^{-2c_{\rm sp}})$.
Thus, uniform spectral decay implies Assumption~\ref{ass:compressibility}
through the HSVD truncation estimate. Such decay is associated with
low-rank tensor-network compressibility, although locality or limited
entanglement growth alone does not guarantee the required uniform
bound.
\end{remark}


For notational simplicity, we suppress the dependence on $\mathbf u$ and
define the truncation perturbation
\begin{equation}
e_k^{(r)}
:=
Z_k-\Pi_r(Z_k),
\qquad
Z_k=F_{\Delta t}(\Psi_k^{(r)},u_k).
\label{eq:truncation_perturbation}
\end{equation}
Then the surrogate dynamics can be written as
\begin{equation}
\Psi_{k+1}^{(r)}
=
F_{\Delta t}(\Psi_k^{(r)},u_k)-e_k^{(r)}.
\label{eq:perturbed_surrogate}
\end{equation}
By Assumption~\ref{ass:compressibility},
$
\|e_k^{(r)}\|
\leq C_{\rm tr}e^{-c_{\rm tr}r},
$
$k=0,\ldots,N_T-1$.

\begin{theorem}[Trajectory Approximation]
\label{thrm:traj-approx} \!\!\!\!
Suppose Assumption~\ref{ass:compressibility} holds and the initial
condition $x_0$ satisfies Assumption~\ref{ass:decay}. Then
$
\max_{0\leq k\leq N_T}
\|\Psi_k-\Psi_k^{(r)}\|
\leq C_T e^{-\bar c r},
\label{eq:trajectory_bound}
$
where
$
C_T:=C_2+N_TC_{\rm tr}$,
$\bar c:=\min\{c',c_{\rm tr}\}$.
In particular, $C_T$ depends on the finite horizon $N_T$ but is
independent of the rank $r$.
\end{theorem}

\begin{proof}
Define $\delta_k:=\|\Psi_k-\Psi_k^{(r)}\|$. Using
\eqref{eq:discrete_flow} and \eqref{eq:perturbed_surrogate},
$
\delta_{k+1}
\leq
\|F_{\Delta t}(\Psi_k,u_k)
-F_{\Delta t}(\Psi_k^{(r)},u_k)\|
+\|e_k^{(r)}\|.
$
By the non-expansiveness property~\eqref{eq:nonexpansive},
$
\delta_{k+1}\leq\delta_k+\|e_k^{(r)}\|,
$
and hence
$
\delta_k
\leq
\delta_0+\sum_{j=0}^{k-1}\|e_j^{(r)}\|.
$
Proposition~\ref{prop:approx} gives
$
\delta_0
=
\|x_0-\Pi_r(x_0)\|
\leq C_2e^{-c'r},
$
while Assumption~\ref{ass:compressibility} gives
$
\|e_j^{(r)}\|
\leq C_{\rm tr}e^{-c_{\rm tr}r}.
$
Therefore,
$
\delta_k
\leq
C_2e^{-c'r}
+kC_{\rm tr}e^{-c_{\rm tr}r}
\leq
(C_2+N_TC_{\rm tr})e^{-\bar c r},
$
where $\bar c=\min\{c',c_{\rm tr}\}$. Taking the maximum over
$k=0,\ldots,N_T$ proves the claim.
\end{proof}


\begin{remark}
Theorem~\ref{thrm:traj-approx} follows a perturbation argument for a
non-expansive system. Unitarity of the Schr\"odinger propagator prevents
amplification of the HT truncation perturbations, which therefore
accumulate additively over the finite horizon. For general dynamics
with Lipschitz constant $L_F$, the recursion becomes
$
\delta_{k+1}\leq L_F\delta_k+\|e_k^{(r)}\|,
$
leading instead to accumulation weighted by powers of $L_F$.
\end{remark}


\subsection{Cost, Value-Function, and Near-Optimality Bounds}

We now quantify the error induced by HT truncation on the cost
functional. Theorem~\ref{thrm:traj-approx}, together with the cost
regularity assumption, yields a uniform bound on the discrepancy
between the full-order and surrogate objectives.

\begin{theorem}[Uniform Cost Approximation]
\label{thm:uniformcostapprox} \!\!\!\!
Under Assum-\\
ptions~\ref{ass:cost_regularity} and
\ref{ass:compressibility}, and assuming that $x_0$ satisfies
Assumption~\ref{ass:decay}, there exists a constant $C_J>0$,
independent of the rank $r$, such that
$
|J(x_0,\mathbf u)-J_r(x_0,\mathbf u)|
\leq C_J e^{-\bar c r}
$
for every $\mathbf u\in\mathcal U^{N_T}$, where
$
C_J=C_T(L_\Phi+N_TL_L)
$
and $\bar c=\min\{c',c_{\rm tr}\}$.
\end{theorem}

\begin{proof}
By the definitions of $J$ and $J_r$, the triangle inequality and
Assumption~\ref{ass:cost_regularity} give
$
|J(x_0,\mathbf u)-J_r(x_0,\mathbf u)|
\leq
L_\Phi\|\Psi_{N_T}-\Psi_{N_T}^{(r)}\|
+
L_L\sum_{k=0}^{N_T-1}
\|\Psi_k-\Psi_k^{(r)}\|.
$
By Theorem~\ref{thrm:traj-approx},
$
\|\Psi_k-\Psi_k^{(r)}\|\leq C_Te^{-\bar c r}
$
for $k=0,\ldots,N_T$. Hence
$
|J(x_0,\mathbf u)-J_r(x_0,\mathbf u)|
\leq
C_T(L_\Phi+N_TL_L)e^{-\bar c r}
=
C_Je^{-\bar c r}.
$
\end{proof}


\begin{proposition}[A Posteriori Cost Certification]
\label{APosteriori}
Suppose Assumption~\ref{ass:cost_regularity} holds.
For any admissible control sequence
$\mathbf u\in\mathcal{U}^{N_T}$, let $e_j^{(r)}$ denote the realized
HT truncation perturbations defined in
\eqref{eq:truncation_perturbation}, and let
$
\delta_0=\|x_0-\Pi_r(x_0)\|.
$
Then
$
|J(x_0,\mathbf u)-J_r(x_0,\mathbf u)|
\leq
L_\Phi
\left(
\delta_0+\sum_{j=0}^{N_T-1}\|e_j^{(r)}\|
\right)
+
L_L
\sum_{k=0}^{N_T-1}
\left(
\delta_0+\sum_{j=0}^{k-1}\|e_j^{(r)}\|
\right),
$
where an empty sum is understood to be zero.
\end{proposition}

\begin{proof}
Let
$
\delta_k=\|\Psi_k-\Psi_k^{(r)}\|.
$
From the trajectory-error recursion in the proof of
Theorem~\ref{thrm:traj-approx},
$
\delta_{k+1}\leq\delta_k+\|e_k^{(r)}\|.
$
Hence
$
\delta_k
\leq
\delta_0+\sum_{j=0}^{k-1}\|e_j^{(r)}\|,
$
$
k=1,\ldots,N_T.
$
Assumption~\ref{ass:cost_regularity} gives
$
|J(x_0,\mathbf u)-J_r(x_0,\mathbf u)|
\leq
L_\Phi\delta_{N_T}
+
L_L\sum_{k=0}^{N_T-1}\delta_k.
$
Substituting the preceding trajectory-error estimate proves the claim.
\end{proof}

Proposition~\ref{APosteriori} provides an a posteriori certificate based
on the truncation perturbations realized during a single HT propagation.
Unlike the a priori bounds of Theorems~\ref{thrm:traj-approx} and
\ref{thm:uniformcostapprox}, it can be evaluated directly from
$\|e_j^{(r)}\|$ and provides a computable bound on the full-order cost
discrepancy for the applied control. Importantly, this a posteriori
estimate does not require Assumption~\ref{ass:compressibility}; hence,
even when uniform rank-dependent compressibility cannot be established,
the realized same-control cost discrepancy can still be certified from
the computed truncation perturbations.


We now quantify the discrepancy between the optimal costs of the
full-order and HT-surrogate problems. The key observation is that
Theorem~\ref{thm:uniformcostapprox} holds uniformly over all admissible
control sequences.

\begin{theorem}[Value-Function Approximation]
\label{thm:valuefuncapprox}
Under Assumptions~\ref{ass:cost_regularity} and
\ref{ass:compressibility}, and assuming that $x_0$ satisfies
Assumption~\ref{ass:decay},
$
|V(x_0)-V_r(x_0)|
\leq C_Ve^{-\bar c r},
$
where $C_V=C_J$ and $\bar c=\min\{c',c_{\rm tr}\}$.
\end{theorem}

\begin{proof}
By Theorem~\ref{thm:uniformcostapprox}, for every
$\mathbf u\in\mathcal U^{N_T}$,
$
J_r(x_0,\mathbf u)-C_Je^{-\bar c r}
\leq J(x_0,\mathbf u)
\leq J_r(x_0,\mathbf u)+C_Je^{-\bar c r}.
$
Taking the infimum over $\mathbf u\in\mathcal U^{N_T}$ yields
$
V_r(x_0)-C_Je^{-\bar c r}
\leq V(x_0)
\leq V_r(x_0)+C_Je^{-\bar c r},
$
which proves the claim with $C_V=C_J$.
\end{proof}




\begin{theorem}[Certified Near-Optimality]
\label{thm:nearopt}
Under Assumptions~\ref{ass:cost_regularity} and
\ref{ass:compressibility}, and assuming that $x_0$ satisfies
Assumption~\ref{ass:decay}, suppose that minimizers
$\mathbf u^\star$ and $\mathbf u_r^\star$ of the full-order and
HT-surrogate problems exist. Then
$
J(x_0,\mathbf u_r^\star)-V(x_0)
\leq C_Ne^{-\bar c r},
$
where $C_N=2C_J$ and $\bar c=\min\{c',c_{\rm tr}\}$.
\end{theorem}

\begin{proof}
By Theorem~\ref{thm:uniformcostapprox},
$
J(x_0,\mathbf u_r^\star)
\leq
J_r(x_0,\mathbf u_r^\star)+C_Je^{-\bar c r}.
$
Since $\mathbf u_r^\star$ minimizes $J_r$,
$
J_r(x_0,\mathbf u_r^\star)
\leq
J_r(x_0,\mathbf u^\star)
\leq
J(x_0,\mathbf u^\star)+C_Je^{-\bar c r},
$
where the second inequality again follows from
Theorem~\ref{thm:uniformcostapprox}. Since
$J(x_0,\mathbf u^\star)=V(x_0)$,
$
J(x_0,\mathbf u_r^\star)-V(x_0)
\leq
2C_Je^{-\bar c r},
$
which proves the claim with $C_N=2C_J$.
\end{proof}

Theorem~\ref{thm:nearopt} shows that controls optimized on fixed-rank
HT surrogates are exponentially near-optimal when evaluated on the
full-order Schr\"odinger dynamics. Thus, increasing the prescribed HT
rank yields a certified decay of the full-order optimality gap.
Realizing computational savings, however, additionally requires a
fully compressed implementation of the underlying tensor-network
operations.

\begin{corollary}[Control--Solution Convergence]
\label{cor:control_convergence}
Under the hypotheses of Theorem~\ref{thm:nearopt}, suppose that
$J(x_0,\cdot)$ is continuous on $\mathcal U^{N_T}$ and that
$\mathbf u^\star$ is the unique minimizer of the full-order problem.
Suppose further that there exist $\mu,\delta>0$ such that
$
J(x_0,\mathbf u)-J(x_0,\mathbf u^\star)
\geq
\frac{\mu}{2}\|\mathbf u-\mathbf u^\star\|^2,
$
for every $\mathbf u\in\mathcal U^{N_T}$ satisfying
$\|\mathbf u-\mathbf u^\star\|\leq\delta$.
Then, for all sufficiently large $r$,
$
\|\mathbf u_r^\star-\mathbf u^\star\|
\leq
\sqrt{2C_N/\mu}\,e^{-\bar c r/2}.
$
\end{corollary}

\begin{proof}
Let
$\mathcal S_\delta:=\{\mathbf u\in\mathcal U^{N_T}:
\|\mathbf u-\mathbf u^\star\|\geq\delta\}$.
If $\mathcal S_\delta=\emptyset$, then
$\|\mathbf u_r^\star-\mathbf u^\star\|<\delta$ for every $r$.
If $\mathcal S_\delta\neq\emptyset$, define
$\gamma_\delta:=\min_{\mathbf u\in\mathcal S_\delta}
\big(J(x_0,\mathbf u)-J(x_0,\mathbf u^\star)\big)$.
Compactness of $\mathcal U^{N_T}$, continuity of $J(x_0,\cdot)$, and
uniqueness of $\mathbf u^\star$ imply $\gamma_\delta>0$.
By Theorem~\ref{thm:nearopt},
$J(x_0,\mathbf u_r^\star)-J(x_0,\mathbf u^\star)
\leq C_Ne^{-\bar c r}$.
Hence, for all sufficiently large $r$,
$C_Ne^{-\bar c r}<\gamma_\delta$, so
$\mathbf u_r^\star\notin\mathcal S_\delta$ and therefore
$\|\mathbf u_r^\star-\mathbf u^\star\|<\delta$.
Thus, in either case, for all sufficiently large $r$ the local
quadratic-growth condition applies, yielding
$\frac{\mu}{2}\|\mathbf u_r^\star-\mathbf u^\star\|^2
\leq J(x_0,\mathbf u_r^\star)-J(x_0,\mathbf u^\star)
\leq C_Ne^{-\bar c r}$.
Consequently,
$\|\mathbf u_r^\star-\mathbf u^\star\|
\leq\sqrt{2C_N/\mu}\,e^{-\bar c r/2}$.
\end{proof}

The local quadratic-growth condition is weaker than global strong
convexity and is therefore suitable for generally nonconvex
quantum-control objectives. Under global $\mu$-strong convexity, it
holds globally and the sufficiently-large-rank qualification is
unnecessary.

\begin{remark}[\!Approximate Surrogate Optimization] \!\!\!\!\!
Suppose\\ 
a numerical optimizer returns
$\widehat{\mathbf u}_r$ satisfying
$
J_r(x_0,\widehat{\mathbf u}_r)
\leq V_r(x_0)+\eta,
$
where $\eta\geq0$ is the surrogate optimization error.
Then
$
J(x_0,\widehat{\mathbf u}_r)-V(x_0)
\leq
2C_Je^{-\bar c r}+\eta.
$
Hence, for a prescribed tolerance $\varepsilon>\eta$, it suffices that
$
r\geq
\left\lceil
\frac{1}{\bar c}
\log\!\left(\frac{2C_J}{\varepsilon-\eta}\right)
\right\rceil
$
to guarantee
$
J(x_0,\widehat{\mathbf u}_r)-V(x_0)\leq\varepsilon.
$
Thus, the full-order optimality gap decomposes into an HT approximation
contribution and a surrogate-optimization contribution.
\end{remark}

\subsection{Rank--Performance Scaling}

We now derive an explicit relation between the prescribed HT rank and
a desired near-optimality tolerance.

\begin{corollary}[Rank--Performance Scaling]
\label{cor:rank_performance}
Suppose the hypotheses of Theorem~\ref{thm:nearopt} hold, so that
$
J(x_0,\mathbf u_r^\star)-V(x_0)
\leq C_Ne^{-\bar c r},
$
where $C_N>0$ and $\bar c>0$ are independent of $r$.
For a desired tolerance $\varepsilon>0$,
$
J(x_0,\mathbf u_r^\star)-V(x_0)\leq\varepsilon
$
is guaranteed whenever
$
r\geq
\left\lceil
\frac{1}{\bar c}
\log\!\left(\frac{C_N}{\varepsilon}\right)
\right\rceil.
$
Consequently, a rank sufficient to guarantee tolerance $\varepsilon$
scales as
$
r=O\!\left(\log(1/\varepsilon)\right).
$
\end{corollary}

\begin{proof}
By Theorem~\ref{thm:nearopt}, it suffices to require
$
C_Ne^{-\bar c r}\leq\varepsilon.
$
Solving for $r$ gives the stated bound, and the logarithmic scaling
follows immediately.
\end{proof}

Corollary~\ref{cor:rank_performance} establishes a theoretical
rank--performance relation under the uniform rank-dependent HT
truncation-accuracy condition. The rank sufficient to guarantee a
prescribed performance tolerance grows only logarithmically with the
inverse tolerance. More rapidly decaying Schmidt spectra across the
tree-induced bipartitions can yield smaller HT truncation errors and
therefore potentially lower sufficient ranks, whereas slowly decaying
spectra may require larger ranks. 
The constants appearing in the sufficient rank bound depend on the
truncation-accuracy and cost-regularity estimates and may be
conservative in practice, since they arise from uniform worst-case
truncation bounds, Lipschitz estimates, and accumulation over the
finite horizon. Moreover, these constants may not be readily available
a priori. Consequently, Corollary~\ref{cor:rank_performance} primarily
establishes the logarithmic rank--performance scaling implied by the
theory rather than providing a directly computable rank-selection rule.
In practice, the realized truncation perturbations can instead be
utilized in the a posteriori estimate of
Proposition~\ref{APosteriori}, yielding a trajectory-specific
performance assessment without replacing each perturbation by its
uniform worst-case bound.


\begin{remark}
The logarithmic rank--performance relation is the optimality counterpart
of the rank--accuracy relation established in~\cite{dehaghani2026certified}.
Under the hypotheses of Theorem~\ref{thm:nearopt}, the full-order
optimality gap of surrogate-optimal controls decreases exponentially
with the prescribed HT rank.
\end{remark}

\begin{remark}[Complexity--Performance Trade-Off]
The sufficient rank bound in Corollary~\ref{cor:rank_performance} also
induces a trade-off between surrogate complexity and control
performance. For a balanced binary dimension tree with bounded local
dimensions and uniform HT rank $r$, a standard HT parametrization
requires $\mathcal{O}(nr^3)$ coefficients in the generic uniform-rank
case, due to the third-order transfer tensors, with lower-order
contributions from the leaf frames. Combining this representation
scaling with
$r=\mathcal{O}(\log(1/\varepsilon))$
from Corollary~\ref{cor:rank_performance} yields
$
\mathcal{O}\!\left(n\log^3(1/\varepsilon)\right)
$
for the corresponding HT representation size.
The computational cost of HT propagation additionally depends on the
tensor-network representation of the Hamiltonian and the chosen
structure-preserving time-stepping algorithm. Therefore, no universal
wall-clock or propagation-cost scaling is asserted here. The numerical
experiments below instead quantify the rank-dependent
representation--performance trade-off and compare the HT representation
size with that of the full-order state.
\end{remark}

\section{Numerical Validation}

We investigate the HT-surrogate optimal-control framework on a controlled
XXZ spin chain, focusing on spectral compressibility, truncation and
trajectory accuracy, and control performance and certification.

\paragraph*{Controlled XXZ Spin-Chain Benchmark}
We consider an $n=8$ spin-$1/2$ XXZ chain with
$H_0=J\sum_{j=1}^{n-1}(X_jX_{j+1}+Y_jY_{j+1}+\Delta Z_jZ_{j+1})$
and global transverse control
$H_1=\frac{1}{n}\sum_{j=1}^{n}X_j$, so that
$H(u_k)=H_0+u_kH_1$, $|u_k|\leq u_{\max}$.
We use $J=1$, $\Delta=0.5$, $u_{\max}=4$, $\rho=10^{-3}$,
$T=3$, $N_T=15$, $\Delta t=0.2$, and $\hbar=1$, with N\'eel
initial state $x_0=\ket{10101010}$ and a balanced binary dimension tree
following the spatial ordering of the spins.
The propagator is unitary, $L_L=0$, and the fidelity terminal cost
satisfies Assumption~\ref{ass:cost_regularity}. Moreover, $x_0$ has
Schmidt rank one and satisfies Assumption~\ref{ass:decay}.
Assumption~\ref{ass:compressibility} is 
assessed numerically below; these finite tests do not establish
uniform validity over the admissible control space.
The target state is generated by the same dynamics using
$u_{\mathrm{ref}}(t_k)=2.2\sin(2\pi t_k/T)
+1.1\sin(4\pi t_k/T+0.4)$, clipped to
$[-u_{\max},u_{\max}]$. The optimization minimizes
$J(x_0,\mathbf u)=
1-|\langle\Psi_{\mathrm{tar}},\Psi_{N_T}(x_0,\mathbf u)\rangle|^2
+\rho\Delta t\sum_{k=0}^{N_T-1}u_k^2$
subject to $|u_k|\leq u_{\max}$ using SLSQP.
Each surrogate and full-order problem is solved from ten initializations.
The best surrogate control is denoted by $\widehat{\mathbf u}_r$, with
$\widehat V_r=J_r(x_0,\widehat{\mathbf u}_r)$, and is reevaluated under
the full-order dynamics; $V_{\mathrm{full}}$ denotes the best full-order
objective obtained. Since global optimality is not established, we use
$V_{\mathrm{ref}}=
\min\{V_{\mathrm{full}},
J(x_0,\widehat{\mathbf u}_r):r\in\mathcal R\}$
as an empirical reference, distinct from the theoretical optimum
$V(x_0)$.

\paragraph*{Hierarchical Spectral Compressibility}
Figure~\ref{fig:spectral_decay}(a) shows hierarchical singular-value
spectra along the best-found full-order trajectory at $k=5,10,15$ for
the balanced $4\vert4$ bipartition. Figure~\ref{fig:spectral_decay}(b)
reports $\max_{k,t}\sigma_{\alpha}^{(t)}(\Psi_k)$ over sampled time
indices and non-root internal nodes. An exponential fit
$Ce^{-c_{\mathrm{fit}}\alpha}$ gives $c_{\mathrm{fit}}=0.153$ and
$R^2=0.950$, indicating approximately exponential decay over the
sampled spectral range. This does not establish
Assumption~\ref{ass:compressibility}, which requires a uniform
rank-dependent truncation bound for pre-truncation surrogate states
over admissible controls and time indices.

\begin{figure}
    \centering
    \includegraphics[width=0.98\linewidth]{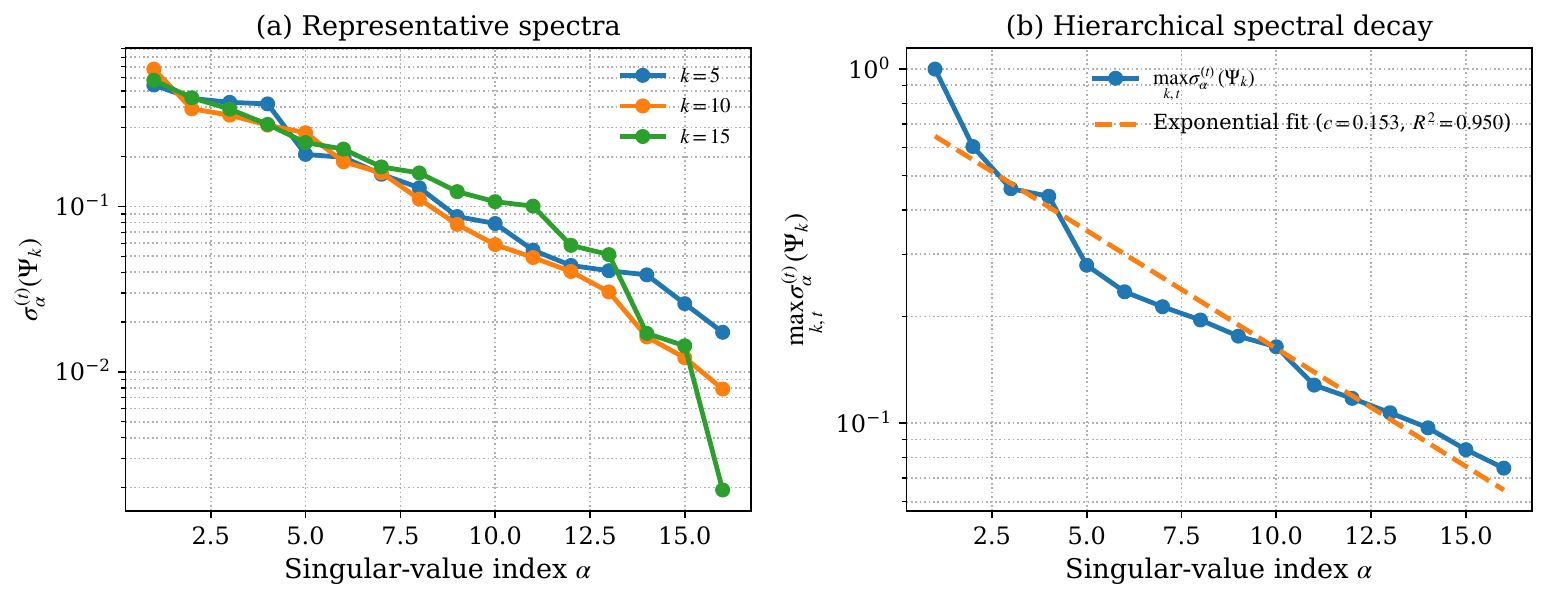}
    \caption{\small
    Hierarchical spectral compressibility along the best-found full-order
    trajectory. (a) Balanced $4\vert4$ spectra at $k=5,10,15$.
    (b) $\max_{k,t}\sigma_{\alpha}^{(t)}(\Psi_k)$ over sampled time
    indices and non-root internal nodes; dashed line:
    $Ce^{-c_{\mathrm{fit}}\alpha}$, with $c_{\mathrm{fit}}=0.153$
    and $R^2=0.950$.}
    \label{fig:spectral_decay}
\end{figure}

\paragraph*{Numerical Stress Test of the HT Truncation Condition}
We assess Assumption~\ref{ass:compressibility} by computing
$E_{\mathrm{tr}}^{\mathrm{samp}}(r)
=\max_{\mathbf u\in\mathcal U_{\mathrm{samp}}}
\max_{0\leq k<N_T}\|e_k^{(r)}(\mathbf u)\|$
over diverse admissible controls, including random, smooth Fourier,
bang--bang, low-amplitude, piecewise-constant, and benchmark-specific
sequences. Independent calibration and validation ensembles of 100
controls each are used.
An exponential envelope fitted to the calibration worst cases yields
$C_{\mathrm{tr}}^{\mathrm{emp}}
e^{-c_{\mathrm{tr}}^{\mathrm{emp}}r}$ with
$c_{\mathrm{tr}}^{\mathrm{emp}}=0.111$.
As shown in Fig.~\ref{fig:assumption3_stress}, validation worst cases
remain below this frozen envelope for all $r=1,\ldots,15$.
This empirical rate describes the tested truncation perturbations and
is distinct from the theoretical
$\bar c=\min\{c',c_{\rm tr}\}$.
Approximate adversarial searches over
$\mathbf u\in[-u_{\max},u_{\max}]^{N_T}$ at
$r\in\{2,4,6,8,10,12,14\}$ also remain below the envelope, with a
maximum adversarial-to-envelope ratio of $0.944$.
These finite tests provide evidence consistent with
Assumption~\ref{ass:compressibility} for this benchmark, but do not
establish uniform validity over the admissible control space.

\begin{figure}
    \centering
    \includegraphics[width=0.5\linewidth]{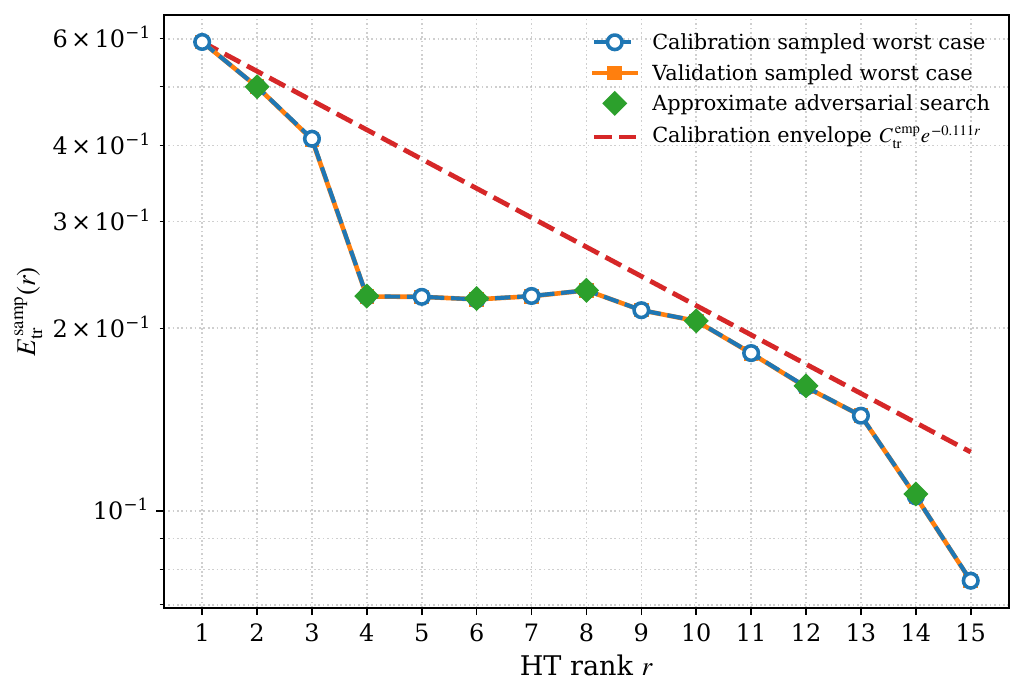}
    \caption{\small
    Stress test of the HT truncation condition. Worst-case perturbations
    for independent calibration and validation ensembles; dashed:
    frozen empirical envelope with
    $c_{\mathrm{tr}}^{\mathrm{emp}}=0.111$; diamonds:
    approximate adversarial searches.}
    \label{fig:assumption3_stress}
\end{figure}

\paragraph*{Rank-Dependent HT Approximation Accuracy}
For uniform HT ranks $r=1,\ldots,16$, each best-found surrogate control
$\widehat{\mathbf u}_r$ is propagated under both the full-order and HT
dynamics. We compute
$E_{\mathrm{traj}}(r)=\max_{0\leq k\leq N_T}
\|\Psi_k(x_0,\widehat{\mathbf u}_r)
-\Psi_k^{(r)}(x_0,\widehat{\mathbf u}_r)\|$
and
$E_{\mathrm{tr}}(r)=\max_{0\leq k<N_T}\|e_k^{(r)}\|$.
Figure~\ref{fig:ht_errors} shows an overall rank-dependent decrease in
both quantities, although monotonicity is not expected.
The full-rank case $r=16$ is omitted because both errors are at numerical
precision and is used only as a consistency check.

\begin{figure}
    \centering
    \includegraphics[width=0.98\linewidth]{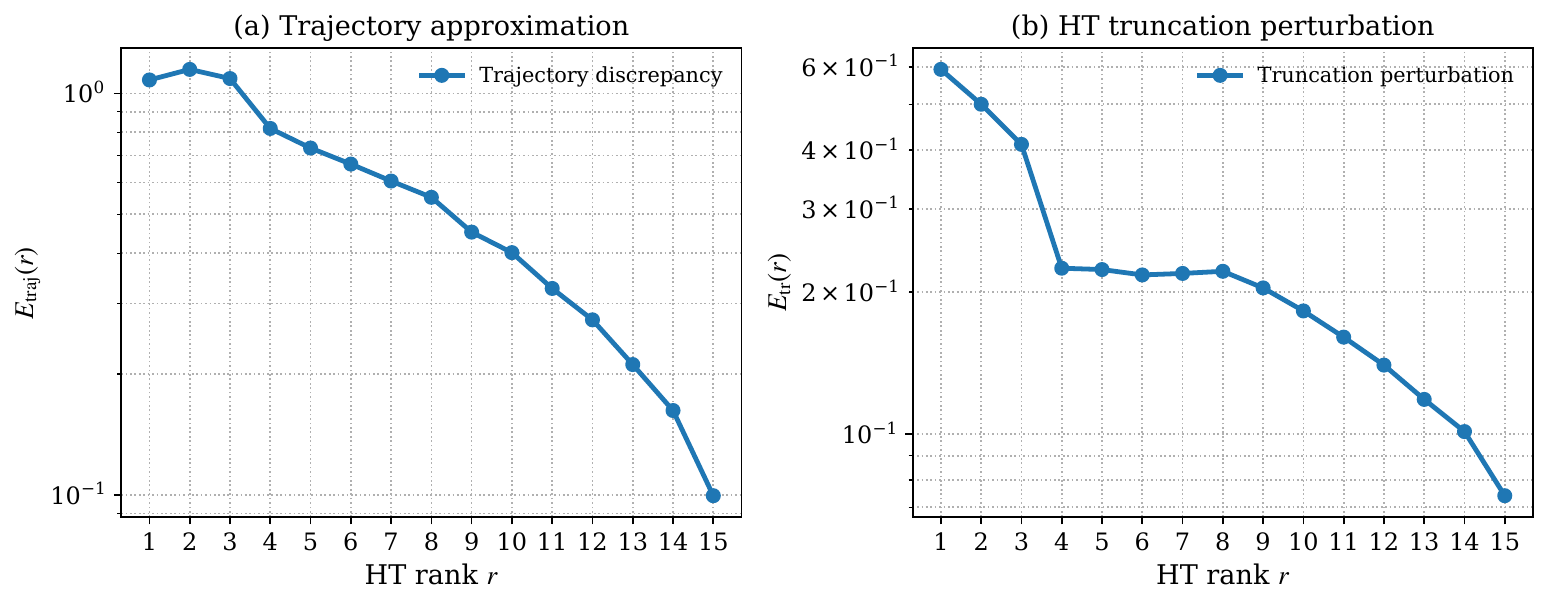}
    \caption{\small
    Rank-dependent HT accuracy for $r=1,\ldots,15$:
    (a) maximum trajectory discrepancy $E_{\mathrm{traj}}(r)$;
    (b) maximum truncation perturbation $E_{\mathrm{tr}}(r)$.}
    \label{fig:ht_errors}
\end{figure}
\paragraph*{Optimal-Control Performance and A Posteriori Certification}

Figure~\ref{fig:optimality_certification}(a) reports the surrogate-value
discrepancy $|\widehat V_r-V_{\mathrm{ref}}|$ and full-order objective
gap $J(x_0,\widehat{\mathbf u}_r)-V_{\mathrm{ref}}$, both showing an
overall reduction with increasing rank. The objective gap need not be
monotone because the surrogate problems are nonconvex and
$\widehat{\mathbf u}_r$ are best-found numerical solutions. Moreover,
these gaps are relative to the empirical $V_{\mathrm{ref}}$, not the
unknown optimum $V(x_0)$.
Figure~\ref{fig:optimality_certification}(b) compares the realized
same-control discrepancy
$E_J(r)=|J(x_0,\widehat{\mathbf u}_r)
-J_r(x_0,\widehat{\mathbf u}_r)|$
with the a posteriori certificate of Proposition~\ref{APosteriori}.
Since $L_L=0$, this discrepancy arises solely from the terminal cost.
The certificate bounds $E_J(r)$ for every tested reduced rank and
generally decreases with rank, although it is conservative at low ranks.
The full-rank case $r=16$, for which the discrepancies are at numerical
precision, is omitted.

\begin{figure}[t]
    \centering
    \includegraphics[width=0.98\linewidth]{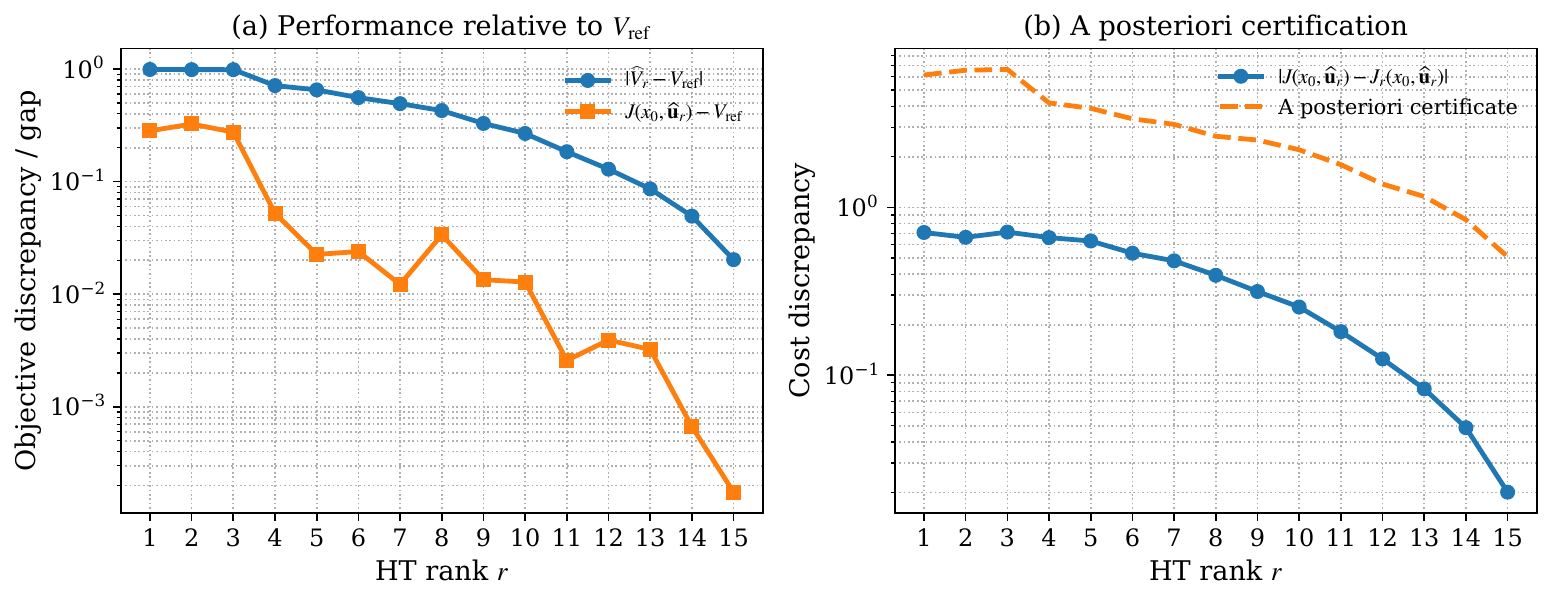}
    \caption{\small
    Control performance and a posteriori certification for $r=1,\ldots,15$.
    (a) $|\widehat V_r-V_{\mathrm{ref}}|$ and
    $J(x_0,\widehat{\mathbf u}_r)-V_{\mathrm{ref}}$.
    (b) Same-control discrepancy $E_J(r)$ and its a posteriori certificate.}
    \label{fig:optimality_certification}
\end{figure}

\paragraph*{Complexity--Performance Trade-Off}
We complement the $n=8$ experiments with an $n=10$ fixed-control study,
propagating three admissible controls under the full-order and HT dynamics
without rank-wise optimization. Figure~\ref{fig:complexity_performance}(a)
compares the canonical HT coefficient count (leaf frames, transfer tensors,
and root core) with the $2^{10}$ full-state coefficients. Accounting for
nodewise rank caps, the HT representation remains smaller through $r=10$
and exceeds the full-state count at $r=12$.
Figure~\ref{fig:complexity_performance}(b) shows the 
same-control discrepancy
$E_J^{(10)}(r;\mathbf u)=|J(x_0,\mathbf u)-J_r(x_0,\mathbf u)|$,
generally decreasing with rank for all three controls. Thus, increasing
rank improves surrogate accuracy while increasing HT representation size.
The full-rank case $r=32$, where the discrepancy is at numerical precision,
is omitted from panel~(b).
These coefficient counts characterize the canonical HT representation,
not measured memory or runtime. Since the present prototype performs
full-space propagation before HT truncation, no computational speedup is
claimed.

\begin{figure}[t]
    \centering
    \includegraphics[width=0.98\linewidth]{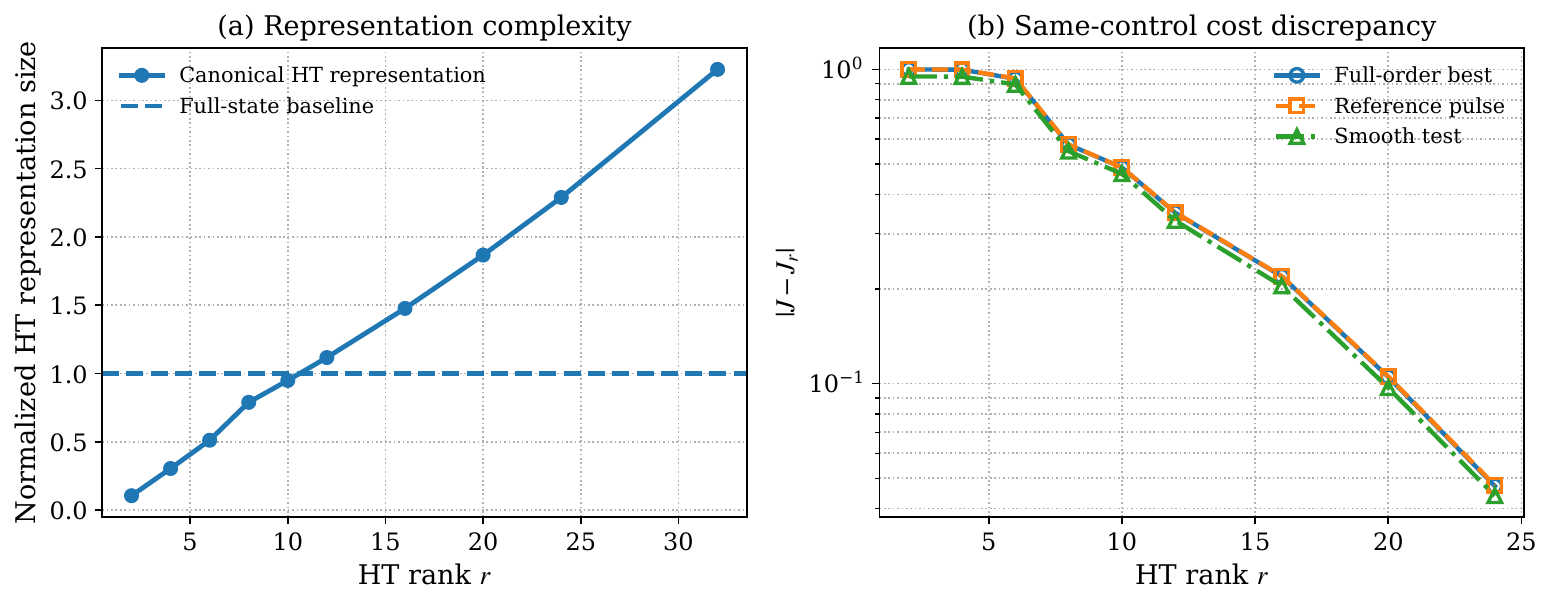}
    \caption{\small
    Complexity--performance trade-off for the $n=10$ fixed-control study.
    (a) HT coefficient count normalized by the $2^{10}$ full-state count.
    (b) Same-control discrepancy $E_J^{(10)}(r;\mathbf u)$ for three
    admissible controls.}
    \label{fig:complexity_performance}
\end{figure}

\section{Conclusion}

We developed a finite-horizon quantum optimal-control framework based on
fixed-rank Hierarchical Tucker surrogates. Under a uniform rank-dependent
HT truncation-accuracy condition, we established rank-dependent trajectory,
cost, and value-function bounds, near-optimality guarantees, and an
a posteriori cost certificate based on realized truncation perturbations.
Numerical experiments on controlled XXZ spin chains provide finite-sample
evidence consistent with the assumed truncation behavior and illustrate
the rank-dependent trade-off between surrogate accuracy, control
performance, and HT representation size. 
Future work will consider fully compressed HT propagation,
adaptive rank selection, and broader many-body and open-system settings.

\bibliographystyle{IEEEtran}

\bibliography{main}

\end{document}